\documentclass{article}
\usepackage[english]{babel}
\usepackage[T1]{fontenc}
\usepackage[color,matrix,arrow]{xy}
\usepackage{tikz}
\usetikzlibrary{decorations.markings,arrows,automata,arrows,backgrounds,snakes}

\usepackage{listings}
\usepackage{color}
\usepackage{algpseudocode}
\usepackage{algorithm}
\usepackage{comment}

\usepackage{amsmath}
\usepackage{amssymb}
\usepackage{amsfonts}
\usepackage{graphicx}
\usepackage[colorinlistoftodos]{todonotes}
\usepackage[colorlinks=true, allcolors=blue]{hyperref}
\usepackage{float}
\usepackage{mathtools}
\usepackage{amssymb,amsmath}
\usepackage{tikz}
\usetikzlibrary{shapes.geometric}
\usepackage{amsthm}
\theoremstyle{definition}

\newcommand{\Aut}    {\mathrm{Aut}}
\newcommand{\Deg}    {\textrm{Deg }}

\newcommand{\Ker}   {\mathrm{Ker}}
\newcommand{\id}   {\mathrm{id}}
\newcommand{\M}    {\mathcal{M}}
\newcommand{\HH}    {\mathcal{H}}
\newcommand{\Hasse}    {\mathcal{H}}

\newtheorem{thm}{Theorem}%[section]

\newcommand{\pd}    {\partial\Delta}

\newtheorem{lem}[thm]   {Lemma}

\newtheorem{rem}[thm]   {Remark}
\newtheorem{defn}[thm]  {Definition}
\newtheorem{ex}[thm]    {Example}
\newtheorem{prop}[thm]  {Proposition}

\newcounter{foo}  \Alph{foo}

\title{On the automorphism group of the Morse complex}
\author{Maxwell Lin, Nicholas A. Scoville}
\date{April 2019}

\begin{document}
\tikzset{->-/.style={decoration={
  markings,
  mark=at position .5 with {{->}}},postaction={decorate}}}

\maketitle
\begin{abstract}
Let $K$ be a finite, connected, abstract simplicial complex.  The Morse complex of $K$, first introduced by Chari and Joswig, is the simplicial complex constructed from all gradient vector fields on $K$.  We show that if $K$ is neither the boundary of the $n$-simplex nor a cycle, then $\Aut(\M(K))\cong \Aut(K)$. In the case where $K= C_n$, a cycle of length $n$, we show that $\Aut(\M(C_n))\cong \Aut(C_{2n})$. In the case where $K=\pd^n$, we prove that $\Aut(\M(\pd^n))\cong \Aut(\pd^n)\times \mathbb{Z}_2$. These results are based on recent work of Capitelli and Minian.
\end{abstract}

\noindent \textit{Keywords:} Discrete Morse theory, automorphism group, Morse complex, gradient vector field\\

\noindent \textit{2000 MSC:}Primary:  55U05, 08A35 Secondary: 52B05,  57Q05

\section{Introduction}

In 2005, Chari and Joswig \cite{CJ-2005} introduced the Morse complex of a simplicial complex.  The Morse complex is based on Forman's discrete Morse theory \cite{F-95, F-02} where after fixing a simplicial complex $K$, one builds a new simplicial complex $\M(K)$ from the collection of all gradient vector fields or arrows on $K$.  Chari and Joswig computed the homotopy type of the Morse complex when $K$ is the $n$-simplex.  Ayala et al. have shown that the pure Morse complex of a tree is collapsible and some other results on the pure Morse complex of an arbitrary graph \cite{A-F-Q-V-08}. Kozlov studied shellability and other properties for trees \cite{Kozlov99}, although the language of the Morse complex was not available to him at the time.  Recently, Capitelli and Minian showed that the isomorphism type of the Morse complex completely determines the isomorphism type of the corresponding simplicial complex \cite{CM-17}.  Other than these results, very little is known about the Morse complex.  Its sheer size alone makes it a notoriously complex object of study.

The goal of this paper is to compute the automorphism group of the Morse complex $\M(K)$.  We derive a formula relating $\Aut(\M(K))$ to $\Aut(K)$ for $K$ any finite, connected, abstract simplicial complex.  Our main result is the following:

\begin{thm}\label{thm: main} Let $K$ be a finite, connected, abstract simplicial complex.  Then
$$
\Aut(\M(K))\cong
\begin{cases}
\Aut(K) & \text{if } K\neq \pd^n,C_n \\
\Aut(C_{2n}) & \text{if } K= C_n \\
\Aut(K)\times \mathbb{Z}_2 & \text{if } K= \pd^n. \\
\end{cases}
$$

\end{thm}

Here $\pd^n$ is the boundary of the $n$-simplex and $C_n$ is the cycle of length $n$. Theorem \ref{thm: main} is proved in three parts.  The first is Proposition \ref{prop: aut(M(K))=aut(K)} where we prove that $\Aut(\M(K))\cong \Aut(K)$ for $K\neq C_n, \pd^n.$ In Section \ref{Induced Maps}, we show that one can induce an automorphism on $\M(K)$ from an automorphism of $K$ . We then show that there is an injection of $\Aut(K)$ into $\Aut(\M(K))$ in Proposition \ref{prop: autK subgroup}. We are then able to show that $\Aut(\M(K))\cong \Aut(K)$ for $K\neq \pd^n, C_n$ by utilizing results of Capitelli and Minian \cite{CM-17}.  These results concern when we may pull an automorphism of the Morse complex back to an automorphism of the original complex.  The case where $K=C_n$ follows as a corollary of the more general fact that $\Aut(\M(K))\cong \Aut(\HH(K))$ where $\HH(K)$ is the Hasse diagram of $K$.  We establish this later isomorphism in Proposition \ref{prop: aut(M(K))=aut(H(K))} and prove that $\Aut(\M(C_n))\cong \Aut(C_{2n})$ in Proposition \ref{prop: aut C_n}.

The case where $K=\pd^n$ is Proposition \ref{prop: aut(M(K)) pt2}. In this case, there are automorphisms of the Morse complex which are not induced by a simplicial map on the original complex, called ghost automorphisms. We define what we call the reflection map $\pi$ which is a cosimplicial map in the sense that if $\alpha\subseteq \beta$, then $\pi(\alpha) \supseteq \pi(\beta)$. This cosimplicial map induces and then generates all the ghost automorphisms of the Morse complex of $\pd^n$. By studying these ghost automorphisms, we account for all automorphisms of $\M(\pd^n)$.

The outline of this paper is as follows:  Section \ref{sec: background} gives necessary background, terminology, and notation. Section \ref{sec: The Automorphism group} is the heart of the paper where we compute the automorphism group of $\M(K)$. In Section \ref{Induced Maps}, we show that any automorphism $K$ induces an automorphisms of $\M(K)$ so that there is an injective homomorphism $\Aut(K) \to \Aut(\M(K))$.  This injection turns out to be an isomorphism in the case where in the case where $K\neq C_n$ or $\pd^n$ by  Proposition \ref{prop: aut(M(K))=aut(K)}. We then turn to the cases $K=C_n$ and $K=\pd^n$.  We show that $\Aut(\M(C_n))\cong \Aut(C_{2n})$ in Section \ref{sec: The Morse complex of the Hasse diagram}.  This follows fairly easily from the fact that $\Aut(\M(K))\cong \Aut(\HH(K))$ (Proposition \ref{prop: aut(M(K))=aut(H(K))}). Finally, Section \ref{sec: Morse complex of pd} is devoted to computing the automorphism group of $\pd^n$ via the ghost automorphisms mentioned above.

\section{Background}\label{sec: background}
All our simplicial complexes are assumed to be abstract, finite, and connected simplicial complexes. Our reference for the basics of simplicial complexes is \cite{F-P11} or \cite{J-11}.

Let $n\geq 1$ be an integer, and write $[v_n]:=\{v_0, v_1, \ldots, v_n\}$. We use $K$ to denote a simplicial complex and $\alpha, \sigma$ etc. to denote a simplex of $K$. If $K$ is a simplicial complex on $n+1$ vertices, the set $V(K):=[v_n]$ is the \textbf{vertex set} of $K$ or the set of $0$-simplices of $K$.  We use $\sigma^{(i)}$ to denote a simplex of dimension $i$, and we write $\tau < \sigma^{(i)}$ to denote any face of $\sigma$ of dimension strictly less than $i$. The number $\dim(\sigma)-\dim(\tau)$ is called the  \textbf{codimension of $\tau$ with respect to $\sigma$}.

\begin{defn}
A \textbf{simplicial map} $f\colon K \to L$ is a function induced by a map on the vertex sets $f_V\colon V(K)\to V(L)$ with the property that if $\sigma=v_{i_0}v_{i_1}\ldots v_{i_m}$ is a simplex in $K$, then $f(\sigma):=f_V(v_{i_0})f_V(v_{i_1})\ldots f_V(v_{i_m})$ is a simplex of $L$.
\end{defn}

If $v_0, v_1, \ldots, v_n$ are all the vertices of a simplex $\sigma$, we will often use the notation $\sigma:=\prod\limits_{i=0}^n v_i$.

\begin{lem}\label{lem: comp}  If $f\colon A\to B$ and $g\colon B\to C$ are simplicial maps, then $(g\circ f)_V = g_V \circ f_V$.
\end{lem}

\begin{proof}
Consider any $\sigma \in V(A)$. Then, $(g\circ f)_V(\sigma) = (g\circ f)(\sigma)$. Since $f,g$ are simplicial maps, we also know that $\dim f(\sigma) = \dim \sigma  =0$, so $f(\sigma) \in V(B)$. Likewise, we know $\dim f(\sigma) = \dim g(f(\sigma))$, so $g(f(\sigma)) \in V(C)$. Therefore, $(g\circ f)_V(\sigma) = (g\circ f)(\sigma) = (g \circ f_V)(\sigma) = (g_V\circ f_V)(\sigma),$ as desired.
\end{proof}

\begin{defn}
A simplicial map which is a bijection is a \textbf{simplicial isomorphism}, and if $f\colon K \to K$ is a simplicial isomorphism, we say that $f$ is a \textbf{simplicial automorphism}. The \textbf{automorphism group} of  $K$ is defined by $$\Aut(K):=\{f\colon K \to K \mid f \text{ is an automorphism}\}.$$
\end{defn}

Because we need to refer to them below, we define a cycle $C_n$ and the boundary of the $n$-simplex $\pd^n.$

\begin{defn}
Let $n \geq 3$ be an integer. Define the \textbf{cycle of length $n$}, denoted $C_n$, to be the $1$-dimensional simplicial complex (graph) with vertex set $V(C_n):=\{v_0, v_1, \ldots, v_{n-1}\}$ and edge set given by
$$\{v_0,v_1\}, \{v_1,v_2\}, \{v_2, v_3\}\ldots, \{v_{n-2}, v_{n-1}\}, \{v_{n-1},v_0\}.$$
\end{defn}

\begin{defn}
Let $n\geq 1$ be an integer.  The \textbf{boundary of the $n$-simplex}, denotes $\pd^n$ is the simplicial complex given by $\pd^n:=\mathcal{P}([v_n])- \{\emptyset, \{v_0,\ldots, v_{n}\}\}.$

\end{defn}

\subsection{The Morse complex}

In this section, we recall the basics of the Morse complex.  Our references for discrete Morse theory in general are \cite{F-02, KnudsonBook, KozlovBook} and the Morse complex in particular are \cite{CJ-2005, CM-17}.

\begin{defn} Let $K$ be a simplicial complex.  A \textbf{discrete vector field} $V$ on $K$ is defined by
$$V:=\{(\sigma^{(p)}, \tau^{(p+1)}) : \sigma< \tau, \text{ each simplex of } K \text{ is in at most one pair}\}.
$$
Any pair in $(\sigma,\tau)\in V$ is called a \textbf{regular pair}, and $\sigma, \tau$ are called \textbf{regular simplices} or just \textbf{regular}.  If $(\sigma^{(p)},\tau^{(p+1)})\in V$, we say that $p+1$ is the \textbf{index} of the regular pair. Any simplex in $K$ which is not in $V$ is called \textbf{critical}.
\end{defn}

\begin{defn}
Let $V$ be a discrete vector field on a simplicial complex $K$.  A \textbf{$V$-path} is a sequence of simplices $$\alpha^{(p)}_0, \beta^{(p+1)}_0, \alpha^{(p)}_1, \beta^{(p+1)}_1, \alpha^{(p)}_2\ldots , \beta^{(p+1)}_{k-1}, \alpha^{(p)}_{k}$$ of $K$ such that $(\alpha^{(p)}_i,\beta^{(p+1)}_i)\in V$ and $\beta^{(p+1)}_i>\alpha_{i+1}^{(p)}\neq \alpha_{i}^{(p)}$ for $0\leq i\leq k-1$. If $k\neq 0$, then the $V$-path is called  \textbf{non-trivial.}
A $V$-path is said to be  \textbf{closed} if $\alpha_{k}^{(p)}=\alpha_0^{(p)}$.
A discrete vector field $V$ which contains no  non-trivial closed $V$-paths is called a \textbf{gradient vector field}.
\end{defn}

\begin{ex}\label{thatExampleV}  An example of a gradient vector field is given on a triangulation of the M\"{o}bius band below:

$$
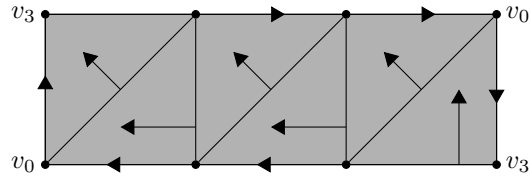
\begin{tikzpicture}[
    decoration={markings,mark=at position 0.6 with {\arrow{triangle 60}}},
    ]

\filldraw[fill=black!30, draw=black] (0,0)--(6,0)--(6,2)--(0,2)--cycle;

\node[inner sep=1pt, circle, fill=black] (a) at (0,0) [draw] {};
\node[inner sep=1pt, circle, fill=black] (b) at (2,0) [draw] {};
\node[inner sep=1pt, circle, fill=black] (c) at (4,0) [draw] {};
\node[inner sep=1pt, circle, fill=black] (d) at (6,0) [draw] {};
\node[inner sep=1pt, circle, fill=black] (e) at (0,2) [draw] {};
\node[inner sep=1pt, circle, fill=black] (f) at (2,2) [draw] {};
\node[inner sep=1pt, circle, fill=black] (g) at (4,2) [draw] {};
\node[inner sep=1pt, circle, fill=black] (h) at (6,2) [draw] {};

\draw[->-]  (b)--(a) node[midway] {};
\draw[->-]  (a)--(e) node[midway] {};
\draw[-]  (a)--(f) node[midway] {};
\draw[->-]  (c)--(b) node[midway] {};
\draw[-]  (b)--(g) node[midway] {};
\draw[-]  (b)--(f) node[midway] {};
\draw[-]  (c)--(d) node[midway] {};
\draw[-]  (c)--(h) node[midway] {};
\draw[-]  (c)--(g) node[midway] {};
\draw[->-]  (h)--(d) node[midway] {};
\draw[-]  (e)--(f) node[midway] {};
\draw[->-]  (f)--(g) node[midway] {};
\draw[->-]  (g)--(h) node[midway] {};

\draw[-triangle 60]  (1,1)--(.5,1.5);
\draw[-triangle 60]  (2,.5)--(1,.5);
\draw[-triangle 60]  (3,1)--(2.5,1.5);
\draw[-triangle 60]  (4,.5)--(3,.5);
\draw[-triangle 60]  (5,1)--(4.5,1.5);
\draw[-triangle 60]  (5.5,0)--(5.5,1);

\node[anchor = north ]  at (a) {};
\node[anchor = north ]  at (b) {};
\node[anchor = north ]  at (c) {};
\node[anchor = north ]  at (d) {};
\node[anchor = south ]  at (e) {};
\node[anchor = south ]  at (f) {};
\node[anchor = south ]  at (g) {};
\node[anchor = south ]  at (h) {};

\node[anchor = east ]  at (a) {{$v_0$}};
\node[anchor = west ]  at (d) {{$v_3$}};
\node[anchor = east ]  at (e) {{$v_3$}};
\node[anchor = west ]  at (h) {{$v_0$}};

\end{tikzpicture}
$$

Another gradient vector field on the  M\"{o}bius band is

$$
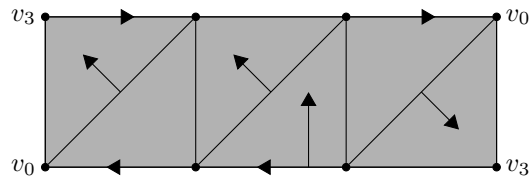
\begin{tikzpicture}[
    decoration={markings,mark=at position 0.6 with {\arrow{triangle 60}}},
    ]

\filldraw[fill=black!30, draw=black] (0,0)--(6,0)--(6,2)--(0,2)--cycle;

\node[inner sep=1pt, circle, fill=black] (a) at (0,0) [draw] {};
\node[inner sep=1pt, circle, fill=black] (b) at (2,0) [draw] {};
\node[inner sep=1pt, circle, fill=black] (c) at (4,0) [draw] {};
\node[inner sep=1pt, circle, fill=black] (d) at (6,0) [draw] {};
\node[inner sep=1pt, circle, fill=black] (e) at (0,2) [draw] {};
\node[inner sep=1pt, circle, fill=black] (f) at (2,2) [draw] {};
\node[inner sep=1pt, circle, fill=black] (g) at (4,2) [draw] {};
\node[inner sep=1pt, circle, fill=black] (h) at (6,2) [draw] {};

\draw[->-]  (b)--(a) node[midway] {};
\draw[-]  (a)--(e) node[midway] {};
\draw[-]  (a)--(f) node[midway] {};
\draw[->-]  (c)--(b) node[midway] {};
\draw[-]  (b)--(g) node[midway] {};
\draw[-]  (b)--(f) node[midway] {};
\draw[-]  (c)--(d) node[midway] {};
\draw[-]  (c)--(h) node[midway] {};
\draw[-]  (c)--(g) node[midway] {};
\draw[-]  (h)--(d) node[midway] {};
\draw[->-]  (e)--(f) node[midway] {};
\draw[-]  (f)--(g) node[midway] {};
\draw[->-]  (g)--(h) node[midway] {};

\draw[-triangle 60]  (1,1)--(.5,1.5);
\draw[-triangle 60]  (3,1)--(2.5,1.5);
\draw[-triangle 60]  (3.5,0)--(3.5,1);
\draw[-triangle 60]  (5,1)--(5.5,.5);

\node[anchor = north ]  at (a) {};
\node[anchor = north ]  at (b) {};
\node[anchor = north ]  at (c) {};
\node[anchor = north ]  at (d) {};
\node[anchor = south ]  at (e) {};
\node[anchor = south ]  at (f) {};
\node[anchor = south ]  at (g) {};
\node[anchor = south ]  at (h) {};

\node[anchor = east ]  at (a) {{$v_0$}};
\node[anchor = west ]  at (d) {{$v_3$}};
\node[anchor = east ]  at (e) {{$v_3$}};
\node[anchor = west ]  at (h) {{$v_0$}};

\end{tikzpicture}
$$

\end{ex}

If a gradient vector field $V$ has only one regular pair, we call $V$ \textbf{primitive}. Given multiple primitive gradient vector fields, we may sometimes combine them to form a new gradient vector field.  This will be accomplished ``overlaying" all the arrows of each primitive gradient vector field.  Clearly such a construction may or may not yield a gradient vector field.

\begin{ex}\label{ex: s1} Let primitive gradient vector fields  $V_0,V_1,V_2$ be given by

$$
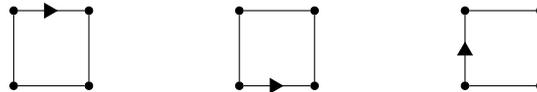
\begin{tikzpicture}[
    decoration={markings,mark=at position 0.6 with {\arrow{triangle 60}}},
    ]

\node[inner sep=1pt, circle, fill=black] (1) at (0,0) [draw] {};
\node[inner sep=1pt, circle, fill=black] (2) at (1,0) [draw] {};
\node[inner sep=1pt, circle, fill=black] (4) at (0,1) [draw] {};
\node[inner sep=1pt, circle, fill=black] (3) at (1,1) [draw] {};

\node[inner sep=1pt, circle, fill=black] (a) at (3,0) [draw] {};
\node[inner sep=1pt, circle, fill=black] (b) at (4,0) [draw] {};
\node[inner sep=1pt, circle, fill=black] (d) at (3,1) [draw] {};
\node[inner sep=1pt, circle, fill=black] (c) at (4,1) [draw] {};

\node[inner sep=1pt, circle, fill=black] (e) at (6,0) [draw] {};
\node[inner sep=1pt, circle, fill=black] (f) at (7,0) [draw] {};
\node[inner sep=1pt, circle, fill=black] (h) at (6,1) [draw] {};
\node[inner sep=1pt, circle, fill=black] (g) at (7,1) [draw] {};

\draw[-]  (1)--(2);
\draw[-]  (2)--(3);
\draw[->-]  (4)--(3);
\draw[-]  (4)--(1);

\draw[->-]  (a)--(b);
\draw[-]  (b)--(c);
\draw[-]  (c)--(d);
\draw[-]  (d)--(a);

\draw[-]  (e)--(f);
\draw[-]  (f)--(g);
\draw[-]  (g)--(h);
\draw[->-]  (e)--(h);

\end{tikzpicture}
$$
respectively.  Then $V_0,V_1$ combine to form a new gradient vector field $V$
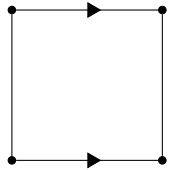
$$
\begin{tikzpicture}[
    decoration={markings,mark=at position 0.6 with {\arrow{triangle 60}}},
    ]

\node[inner sep=1pt, circle, fill=black] (1) at (0,0) [draw] {};
\node[inner sep=1pt, circle, fill=black] (2) at (2,0) [draw] {};
\node[inner sep=1pt, circle, fill=black] (4) at (0,2) [draw] {};
\node[inner sep=1pt, circle, fill=black] (3) at (2,2) [draw] {};

\draw[->-]  (1)--(2);
\draw[-]  (2)--(3);
\draw[->-]  (4)--(3);
\draw[-]  (4)--(1);

\end{tikzpicture}
$$
but clearly combining $V_1$ and $V_2$
$$
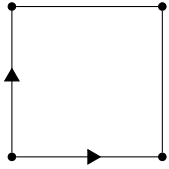
\begin{tikzpicture}[
    decoration={markings,mark=at position 0.6 with {\arrow{triangle 60}}},
    ]

\node[inner sep=1pt, circle, fill=black] (1) at (0,0) [draw] {};
\node[inner sep=1pt, circle, fill=black] (2) at (2,0) [draw] {};
\node[inner sep=1pt, circle, fill=black] (4) at (0,2) [draw] {};
\node[inner sep=1pt, circle, fill=black] (3) at (2,2) [draw] {};

\draw[->-]  (1)--(2);
\draw[-]  (2)--(3);
\draw[-]  (4)--(3);
\draw[->-]  (1)--(4);

\end{tikzpicture}
$$
is not a gradient vector field since the bottom left vertex is in two pairs of $V$.
\end{ex}

If $V,W$ are two gradient vector fields, write $V\leq W$ whenever the regular pairs of $V$ are also regular pairs of $W$.  In general, we say that a collection of primitive vector fields $V_0, V_1, \ldots, V_n$ is \textbf{compatible} if there exists a gradient vector field $V$ such that $V_i\leq V$ for all $0\leq i\leq n$.

\begin{defn}\label{MorseComplexDef2}  The  \textbf{Morse complex} of $K$, denote $\mathcal{M}(K)$, is the simplicial complex whose vertices are given by primitive gradient vector fields and whose $n$-simplices are given by gradient vector fields with $n+1$ regular pairs.  A gradient vector field $V$ is then associated with all primitive gradient vector fields $V:=\{V_0, \ldots, V_n\}$ with $V_i\leq V$ for all $0\leq i\leq n$.
\end{defn}

\begin{ex} Let $K=\pd^1=C_3$ be the simplicial complex given by

$$
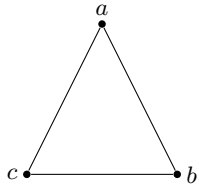
\begin{tikzpicture}[scale=1]

\node[inner sep=1pt, circle, fill=black](a) at (0,2) {};
\node[inner sep=1pt, circle, fill=black](b) at (1,0) {};
\node[inner sep=1pt, circle, fill=black](c) at (-1,0) {};

\draw[-]  (a)--(b);
\draw[-]  (a)--(c);
\draw[-]  (b)--(c);

\node[anchor = south]  at (a) {\small{$a$}};
\node[anchor =  west]  at (b) {\small{$b$}};
\node[anchor = east]  at (c) {\small{$c$}};

\end{tikzpicture}
$$

Here we adopt the convention that if $(x,xy)$ is a primitive vector field consisting of a vertex and edge, we denote this by $xy$.  Note that this notation only works for a primitive vector of index $1$. Then one checks that the Morse complex $\M(K)$ is given by:

$$
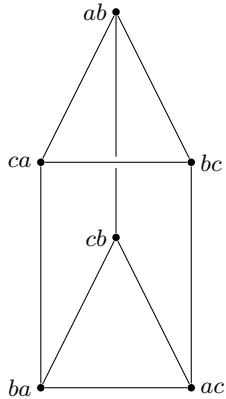
\begin{tikzpicture}[scale=1]

\node[inner sep=1pt, circle, fill=black](a) at (0,2) {};
\node[inner sep=1pt, circle, fill=black](b) at (1,0) {};
\node[inner sep=1pt, circle, fill=black](c) at (-1,0) {};
\node[inner sep=1pt, circle, fill=black](d) at (0,-1) {};
\node[inner sep=1pt, circle, fill=black](e) at (1,-3) {};
\node[inner sep=1pt, circle, fill=black](f) at (-1,-3) {};

\draw[-]  (a)--(b);
\draw[-]  (a)--(c);
\draw[-] (a)--(d);
\draw [white, line width=1.5mm] (b) -- (c) node[midway, left] {};
\draw[-]  (b)--(c);

\draw[-]  (d)--(e);
\draw[-]  (d)--(f);
\draw[-]  (e)--(f);

\draw[-] (b)--(e);
\draw[-] (c)--(f);

\node[anchor = east]  at (a) {\small{$ab$}};
\node[anchor =  west]  at (b) {\small{$bc$}};
\node[anchor = east]  at (c) {\small{$ca$}};
\node[anchor = east]  at (d) {\small{$cb$}};
\node[anchor =  west]  at (e) {\small{$ac$}};
\node[anchor = east]  at (f) {\small{$ba$}};

\end{tikzpicture}
$$

\end{ex}

\section{The Automorphism group of $\M(K)$}\label{sec: The Automorphism group}

This section is devoted to computing the automorphism group of $\M(K)$. We first show in section \ref{Induced Maps} that certain automorphisms of $K$ induce automorphisms of $\M(K)$.  It will then follow by Proposition \ref{prop: autK subgroup} that $\Aut(K)$ is isomorphic to a subgroup of $\Aut(\M(K))$.  We show in Proposition \ref{prop: aut(M(K))=aut(K)} that $\Aut(K)\cong \Aut(\M(K))$ in the case where $K\neq C_n$ or $\pd^n$. The next two sections compute $\Aut(\M(K))$ in the two excluded cases.

\subsection{Induced maps on $\M(K)$}\label{Induced Maps}

If $f_V\colon V(K) \to V(L)$ is a  bijection on the vertex sets and induces a simplicial map $f\colon K\to L$, then $f$ is an isomorphism. We are interested in isomorphisms from $K$ to $K$, i.e., automorphisms of $K$. Given an automorphism of $K$, we now define an induced automorphism on $\M(K)$.

% \begin{prop} Let $K$ be a simplicial complex. If $f_V: V(K) \to V(K)$ is bijective and induces a simplicial map $f: K\to K$, then $f$ is an isomorphism.
% \end{prop}
% \begin{proof}
% %slight abuse of notation with product notation
% Since $f_v$ is bijective, there exists a function $g_V: V(K)\to V(K)$ such that $f_V \circ g_V = \text{id}_{V(K)}$ and $g_V \circ f_V = \text{id}_{V(K)}$. Then, $g_V$ must also induce a function $g: K\to K$. We now show $g$ is a simplicial map. For any $\sigma =\prod_{j=0}^m v_{i_j} \in K$, since $f_V$ is bijective there exist $v_{h_0}, v_{h_1}, \cdots v_{h_j} \in V(K)$ such that $f_V(v_{h_n}) = v_{i_n}$ for all $0\leq n \leq j$. Thus, we can express $\sigma = \prod_{j=0}^m f_V(v_{h_j})$. Then, we have $g(\sigma) =  \prod_{j=0}^m g_V(f_V(v_{h_j})) =\prod_{j=0}^m v_{h_j}  \in K$, so $g$ is a simplicial map.\\
% For any simplex $\tau =\prod_{j=0}^m v_{i_j} \in K$, we have $g\circ f (\tau) = g\left(\prod_{j=0}^m f_V(v_{i_j})\right) = \prod_{j=0}^m g_V(f_V(v_{i_j})) = \prod_{j=0}^m v_{i_j} = \tau$. Similarly, $f\circ g(\tau) = \tau$. Therefore, $g\circ f = f\circ g = \text{id}_K$, as desired.
% \end{proof}

\begin{defn}
Let  $f\colon K\to K$ be an automorphism.  Define the \textbf{induced automorphism on the Morse complex}  $f_{\ast_V}\colon V(\mathcal{M}(K)) \to V(\mathcal{M}(K))$ by  $f_{\ast_V}(v) = (f(\sigma), f(\tau))$ where $v = (\sigma^{(p)}, \tau^{(p+1)}) \in V(\mathcal{M}(K))$.
\end{defn}

We then extend $f_{\ast_V}$ to a simplicial map on all of $\M(K)$. Below we will justify our claims that this yields a well-defined automorphism of $\M(K)$.

\begin{prop}\label{prop: primitive to primitive} If $v =(\sigma^{(p)}, \tau^{(p+1)})  \in V(\mathcal{M}(K))$, then $f_{\ast_V}(v) \in V(\mathcal{M}(K))$.
\end{prop}

\begin{proof}
We seek to show that $(f(\sigma), f(\tau))$ is a primitive vector of $K$. Since $f_V$ is a bijection, we know that $\dim f(\sigma) = \dim \sigma$ and $\dim f(\tau) = \dim \tau$. Hence we have $\dim f(\sigma) = \dim\sigma = \dim\tau-1 =\dim f(\tau) - 1$. Since $\sigma \subseteq \tau$, $f(\sigma)\subseteq f(\tau)$ so that $(f(\sigma), f(\tau))$ is primitive and $(f(\sigma), f(\tau))\in V(\M(K)).$
\end{proof}

%We will now investigate what happens when $f_\ast$ is a simplicial map.
\begin{prop}\label{prop: f_* simplicial}
Let $f\colon K \to K$ be a simplicial automorphism. Then the induced map $f_{\ast}$ is simplicial.
\end{prop}

\begin{proof}
Suppose for the sake of contradiction that $f_{\ast}\colon \mathcal{M}(K)\to \mathcal{M}(K)$ is not a simplicial map. Then there is some $\alpha = v_{i_0}v_{i_1}\cdots v_{i_m}\in \mathcal{M}(K)$ such that $f_{\ast} (\alpha) = f_{\ast_V}(v_{i_0}) f_{\ast_V}(v_{i_1})\cdots  f_{\ast_V}(v_{i_m})\not\in \mathcal{M}(K)$.  This implies that the induced function $f_\ast(\alpha)$ either does not induce a gradient vector field (i.e. two vertices in $f_\ast(\alpha)$ are not compatible), or it induces a discrete vector field that contains a nontrivial closed $V$-path.\\

\textbf{Case 1:} Suppose for the sake of contradiction that there are at least two distinct simplices, say $v_{i_j} = (\sigma_j, \tau_j)$ and $v_{i_k} = (\sigma_k, \tau_k)$, which are incompatible. This implies that exactly one of the following holds:

\begin{enumerate}
\item[(1)] $f(\sigma_i) = f(\sigma_j)$
\item[(2)]$f(\tau_i) = f(\tau_j)$
\item[(3)]$f(\tau_j) = f(\sigma_k)$
\item[(4)]$f(\sigma_j) = f(\tau_k)$.
\end{enumerate}

Suppose $(1)$ is true. Since $f$ is an automorphism, this implies $\sigma_i = \sigma_j$. However, since $\tau_i \neq \tau_j$, we cannot simultaneously have both $(\sigma_i, \tau_i)$ and $(\sigma_i, \tau_j)$, so this would imply $\alpha \not\in \mathcal{M}(K)$, a contradiction. The cases in which $(2)$, $(3)$, or $(4)$ are true lead to similar contradictions, again due to the fact that $f$ is a simplicial automorphism.\\

\textbf{Case 2:} Suppose for the sake of contradiction that $f_\ast(\alpha)$ contains a nontrivial closed V-path. Then, $f_\ast(\alpha)$ contains some vertices $$f_\ast(\sigma_0, \tau_0), f_\ast(\sigma_1, \tau_1), f_\ast(\sigma_2, \tau_2), \dots , f_\ast(\sigma_{k}, \tau_{k})$$ $$ = (f(\sigma_0), f(\tau_0)), (f(\sigma_1), f(\tau_1)), (f(\sigma_2), f(\tau_2)), \dots , (f(\sigma_{k-1}), f(\tau_{k-1}))$$ such that $f(\sigma_0), f(\tau_0), f(\sigma_1), f(\tau_1), \dots , f(\sigma_{k-1}), f(\tau_{k-1}), f(\sigma_0)$ is a closed V-path, in which $f(\sigma_i) < f(\tau_{i-1})$ for $0 \leq i \leq k-1$ (where the indices are taken mod $k$). Since $f$ is an automorphism, it follows that $\sigma_0, \tau_0, \sigma_1, \tau_1, \dots , \sigma_{k-1}, \tau_{k-1}, \sigma_0$ must also be a nontrivial closed V-path. Since $(\sigma_0, \tau_0), (\sigma_1, \tau_1), \dots , (\sigma_{k-1},\tau_{k-1})$ are vertices in $\alpha$, it follows that $\alpha \not\in \M(K)$, a contradiction. Therefore, $f_\ast(\alpha) \in \M(K)$, so $f_{\ast}$ must be a simplicial map.
\end{proof}

\begin{ex}
We give an example to show that if the simplicial map $f\colon K \to K$ is not an automorphism, then $f_*$ is not necessarily a simplicial map.  Indeed, consider the simplicial complex $K = \Delta^2$ with vertex set $\{a,b,c\}$. Define $f_V\colon V(K)\to V(K)$ by $f_V(v) = a$ for all vertices $v$. It is easy to verify that $f$ is a simplicial map. Let $f_\ast\colon \M(K) \to \M(K)$ be induced by $f$. Notice that $(a,ab) \in \M(K)$, but $f_\ast((a,ab)) = (f(a), f(ab)) = (a,a) \not\in \M(K)$. In order to avoid these ``degenerate" primitive vectors, we must impose the constraint that $f$ be an isomorphism.

\end{ex}

We now show that a simplicial automorphism on a simplicial complex $K$ gives rise to a simplicial automorphism on $\M(K)$.

\begin{prop}\label{prop: induced auto} Let $f\colon K\to K$ be a simplicial automorphism. Then the induced map $f_{\ast}\colon \mathcal{M}(K) \to \mathcal{M}(K)$ is a simplicial automorphism.
\end{prop}

\begin{proof} We know by Proposition \ref{prop: f_* simplicial} that $f_*$ is simplicial.  It thus suffices to show that $f_{\ast_V}\colon V(\mathcal{M}(K)) \to V(\mathcal{M}(K))$ is a bijection. Consider any $(\sigma, \tau)  \in V(\mathcal{M}(K))$. We know that $f$ is a simplicial isomorphism, so it has an inverse $g\colon K \to K$. Thus, $(g(\sigma), g(\tau)) \in V(\mathcal{M}(K))$, as an argument similar to that in Proposition \ref{prop: primitive to primitive} verifies that $(g(\sigma), g(\tau))$ is primitive. Hence, $f_{\ast_V}((g(\sigma), g(\tau))) = (f\circ g(\sigma), f\circ g(\tau)) = (\sigma, \tau)$. Therefore, $f_{\ast_V}$ is surjective. Since $V(\mathcal{M}(K))$ is finite, this implies $f_{\ast_V}$ is bijective.
\end{proof}

We can then show that the induced map respects composition.

\begin{lem} \label{lem: induced comp} Let $f,g\colon K\to K$ be simplicial automorphisms with induced automorphisms $f_\ast, g_\ast \colon \M(K)\to \M(K)$, respectively. Then $(f\circ g)_\ast = f_\ast \circ g_\ast$.
\end{lem}
\begin{proof}
It suffices to show that $(f\circ g)_{\ast_V} = (f_{\ast} \circ g_{\ast})_V$. Consider any $(\sigma, \tau) \in \M(K)$. We have
\begin{align*}
(f\circ g)_{\ast_V}((\sigma,\tau)) &= (f\circ g(\sigma), f\circ g(\tau))\\
&= f_{\ast_V}(g(\sigma), g(\tau))\\
&= (f_{\ast_V} \circ g_{\ast_V})((\sigma, \tau)).
\end{align*}
Hence, $(f\circ g)_{\ast_V} = f_{\ast_V}\circ g_{\ast_V}$. By Lemma \ref{lem: comp}, we know that $f_{\ast_V} \circ g_{\ast_V} = (f_{\ast} \circ g_{\ast})_V $, so $(f\circ g)_{\ast_V} = (f_{\ast} \circ g_{\ast})_V$.

\end{proof}

Because automorphisms of $K$ induce automorphisms of $\M(K)$, we next show that we are able to obtain not only elements of $\Aut(\M(K))$, but that $\Aut(K)$ actually corresponds to a subgroup of $\Aut(\M(K)).$

\begin{prop}\label{prop: autK subgroup}
The group $\text{Aut}(K)$ is isomorphic to a subgroup of $\text{Aut}(\mathcal{M}(K))$. In particular, there is an injective homomorphism $\phi\colon \Aut(K)\to \Aut(\M(K)).$
\end{prop}

\begin{proof}
Consider the function $\phi\colon \text{Aut}(K) \to \text{Aut}(\mathcal{M}(K))$ that sends each simplicial automorphism $f \in \text{Aut}(K)$ to its induced simplicial automorphism $f_{\ast}\in \text{Aut}(\mathcal{M}(K))$. We first show that $\phi$ is a homomorphism. Consider any $f,g \in \text{Aut}(K)$, and an arbitrary $\sigma = \prod_{i=0}^m(\sigma_i, \tau_i) \in \mathcal{M}(K)$ where each $(\sigma_i,\tau_i)$ is a vertex of $\sigma$. Since $f,g\in \Aut(K)$, $f\circ g\in \Aut(K)$. This induces $(f\circ g)_{\ast}\colon \mathcal{M}(K) \to \mathcal{M}(K)$. Notice that

\begin{align*}
\phi(f\circ g)(\sigma)  &=(f\circ g)_{\ast}(\sigma) \\
&=  \prod_{i=0}^m (f\circ g)_{\ast_V}((\sigma_i, \tau_i))  \\
&= \prod_{i=0}^m((f\circ g)(\sigma_i), (f\circ g)(\tau_i)).
\end{align*}

On the other hand, we have

\begin{align*}
(\phi(f) \circ \phi(g))(\sigma) &= (f_\ast \circ g_\ast)(\sigma) \\
&= f_\ast\left(\prod_{i=0}^mg_{\ast_V}((\sigma_i, \tau_i)) \right) \\
&= f_\ast\left(\prod_{i=0}^m(g(\sigma_i), g(\tau_i))\right) \\
&= \prod_{i=0}^m((f\circ g)(\sigma_i), (f\circ g)(\tau_i))
\end{align*}

Thus $\phi(f \circ g) = \phi(f)\circ \phi(g)$.\\

We now need to show that $\phi$ is injective. Consider any $f \in \Ker(\phi)$. We claim that $f = \id_K$. Again, consider an arbitrary $\sigma = \prod_{i=0}^m(\sigma_i, \tau_i) \in \mathcal{M}(K)$. Notice $f$ induces a simplicial map $f_\ast \colon \M(K) \to \M(K)$. Since $f \in \Ker(\phi)$, we have
\begin{align*}
\phi(f)(\sigma) &= \sigma\\
f_\ast (\sigma) &= \sigma\\
\prod_{i=0}^m f_{\ast_V}((\sigma_i,\tau_i)) &= \prod_{i=0}^m (\sigma_i,\tau_i)\\
\prod_{i=0}^m (f(\sigma_i),f(\tau_i)) &= \prod_{i=0}^m (\sigma_i,\tau_i).
\end{align*}
As this holds for any choice of $\sigma$, we conclude that $f = \id_K$. Since $\Ker(\phi)$ is trivial, it follows that $\phi$ is injective.

\end{proof}

Proposition \ref{prop: autK subgroup} guarantees that $\Aut(K)\leq \Aut(\M(K))$ by inducing an automorphism of $\M(K)$ from an automorphism of $K$. Thus if every automorphism of $\M(K)$ is induced by an automorphism of $K$, then $\Aut(\M(K))=\Aut(K).$  We will first show that this is indeed the case for $K\neq C_n, \pd^n$.

% \begin{thm}\cite[Theorem 3.5]{CM-17} Let $G, G'\neq C_n$ be connected simple graphs and let $F\colon \M(G)\to \M(G')$ be a simplicial isomorphism.  Define $f\colon G \to G'$ by $f(v):=s(F(v,e)),$ where $e$ is any edge incident to $v$ and $s(F(v,e))$ denotes the source of $F(v,e)$. Then $f$ is a well-defined simplicial isomorphism.
% \end{thm}

\begin{prop}\label{prop: aut(M(K))=aut(K)}  If $K$ is a simplicial complex other than $C_n$ or $\pd^n$, then $\Aut(\M(K)) \cong \Aut(K)$ .
\end{prop}

\begin{proof}
As in Proposition \ref{prop: autK subgroup}, define a function $\phi\colon \text{Aut}(K) \to \text{Aut}(\mathcal{M}(K))$ that sends each simplicial automorphism $f \in \text{Aut}(K)$ to its induced simplicial automorphism $f_{\ast}\in \text{Aut}(\mathcal{M}(K))$. We know that $\phi$ is an injective homomorphism by Proposition \ref{prop: autK subgroup}.

We now show that there is a surjection onto $\Aut(\M(K))$.  If so, since $\M(K)$ and $K$ are finite, this implies that $\phi$ is an isomorphism.  Let $F\in \Aut(\M(K)).$  If $K$ is a 1-dimensional simplicial complex (i.e., a graph) other than $C_n$, then $F$ is induced by a simplicial isomorphism $f\colon K\to K$ by \cite[Theorem 3.5]{CM-17}.  Thus the result for $K$ a graph other than $C_n$.

Now suppose that $K\neq \pd^n$ is a simplicial complex of dimension greater than or equal to $2$, and let $F\in \Aut(\M(K)).$ In the proof of Theorem A \cite{CM-17}, Capitelli and Minian construct a simplicial isomorphism $f\colon K \to K$ that induces the given $F\colon \M(K)\to \M(K).$  Their construction of $f$ relies on a condition that is not satisfied for $K=\pd^n$ according to the contrapositive of \cite[Theorem 4.2]{CM-17}. Thus the result for $K\neq \pd^n$ a simplicial complex of dimension greater than $2$.
\end{proof}

% \begin{conj} If $\varphi: \mathcal{M}(K) \to \mathcal{M}(K)$ is an isomorphic simplicial map, then $\varphi_V$ is induced by some simplicial map $f: K\to K$.
% \end{conj}

\subsection{The Morse complex of the Hasse diagram}\label{sec: The Morse complex of the Hasse diagram}

In this section, we will show that computing the automorphism group of $\M(K)$ is equivalent to computing the automorphism group of $\HH(K)$.  We will then immediately be able to compute the Morse complex of $C_n$, one of the two cases excluded in Proposition \ref{prop: aut(M(K))=aut(K)}.  Section \ref{sec: Morse complex of pd} is then devoted to computing the Morse complex of our final special case, $K=\pd^n.$  We briefly recall here the definition of the Hasse diagram.

\begin{defn}
The \textbf{Hasse diagram} of $K$, denoted $\mathcal{H}_K$ or $\mathcal{H}$, is defined as the partially ordered set of simplices of $K$ ordered by the face relations. We view $\HH$ as a graph.
\end{defn}

We adopt the convention that if $\sigma^{(p)}<\tau^{(p+1)}$ are two nodes of the Hasse diagram, the edge joining them is denoted $\sigma\tau$.

\begin{prop} \label{prop: aut(M(K))=aut(H(K))}  For any simplicial complex $K$, $\Aut(\M(K)) \cong \Aut(\Hasse(K))$.
\end{prop}
\begin{proof}
%We introduce some notation for the Hasse diagram $\Hasse(K)$. Call the set of its edges $E(\Hasse(K))$.\\

We will construct an isomorphism $\phi\colon \Aut(\Hasse(K)) \to \Aut(\M(K))$. Consider an arbitrary automorphism $f \in \Aut(\Hasse(K)) $. Define a function \\ $m \colon E(\Hasse(K)) \to V(\M(K))$ by $m(\sigma\tau) = (\sigma, \tau)$, where $\sigma$ is a codimension $1$ face of $\tau$. Notice that $m$ has an inverse $m^{-1} \colon V(\M(K)) \to E(\Hasse(K))$ given by $m^{-1}(\sigma, \tau)=\sigma\tau$.
Define a function $g_V\colon V(\M(K)) \to V(\M(K))$ by $g_V  := m\circ f \circ m^{-1}$. Clearly  $g_V((\sigma,\tau))$ is another vertex of $V(\M(K))$. Since $f$ is a simplicial isomorphism, it has an inverse $f^{-1} \colon \Hasse(K) \to \Hasse(K)$. It is then clear that $g_V$ has an inverse $g_V^{-1} := m^{-1} \circ f^{-1} \circ m$. Hence $g_V$ is a bijection.\\
We now show that $g$ is simplicial.  Consider any $\sigma  := (\sigma_1, \tau_1)(\sigma_2,\tau_2)\cdots (\sigma_k,\tau_k)\in \M(K)$. Observe that
\begin{align*}
g(\sigma) &= (m\circ f\circ m^{-1})(\sigma) \\
&= (m\circ f)\left( \prod_{i=1}^k \sigma_i\tau_i\right)\\
&= m\left( \prod_{i=1}^k f_V(\sigma_i) f_V(\tau_i)\right)\\
&= \prod_{i=1}^k m(f_V(\sigma_i)f_V(\tau_i)).
\end{align*}

Suppose that $\prod_{i=1}^k m(f_V(\sigma_i)f_V(\tau_i))= \prod_{i=1}^k (f_V(\sigma_i), f_V(\tau_i))$ as the case $\prod_{i=1}^k m(f_V(\sigma_i)f_V(\tau_i))= \prod_{i=1}^k (f_V(\tau_i), f_V(\sigma_i))$ is identical. We claim that each $(f_V(\sigma_i), f_V(\tau_i))$ is a primitive vector in $\M(K)$.  Since $\sigma_i\tau_i \in \HH(K)$, by the construction of $\HH(K)$, we know that either $\sigma_i$ is a codimension $1$ face of $\tau_i$, or vice versa, and since $f$ is an automorphism, $f_V(\sigma_i)f_V(\tau_i) \in \HH(K)$. Again, by construction of $\HH(K)$, we know that either $f_V(\sigma_i)$ is a codimension $1$ face of $f_V(\tau_i)$, or vice versa. Thus, applying $m$ to $f_V(\sigma_i)f_V(\tau_i)$ will form a primitive vector $(f_V(\sigma_i)f_V(\tau_i))\in \M(K)$. \\
We now must show that all primitive vectors in $g(\sigma) = \prod_{i=1}^k (f_V(\sigma_i), f_V(\tau_i))$ are compatible. Suppose for the sake of contradiction they are not all compatible. The first possibility is that there are two compatible primitive vectors $(\sigma_k,\tau_k),(\sigma_{\ell}, \tau_{\ell}) \in V(\M(K))$ such that

$$(f_V(\sigma_k),f_V(\tau_k))(f_V(\sigma_{\ell}),f_V(\tau_{\ell})) \not\in \M(K).$$
However, this implies that $\{f_V(\sigma_k),f_V(\tau_k)\} \cap \{f_V(\sigma_{\ell}),f_V(\tau_{\ell})\} \neq \emptyset$. Since $f_V$ is bijective, this implies that $\{\sigma_k,\tau_k\} \cap \{\sigma_{\ell}, \tau_{\ell}\} \neq \emptyset$, so that $(\sigma_k,\tau_k),(\sigma_{\ell}, \tau_{\ell})$ are not compatible, which is a contradiction. The second possibility is that $g(\sigma)$ contains a nontrivial closed $V$-path.  In a similar manner to the first case, this nontrivial closed $V$-path in the image of $g$ can be pulled back to obtain a nontrivial closed $V$-path in $\sigma$, a contradiction.  We conclude that $g(\sigma)\in \M(K)$ so that $g$ is a simplicial map.

% The second possibility is that $g(\sigma)$ contains a nontrivial closed V-path, consisting of the vertices $$(f_V(\alpha_1),f_V(\beta_1)) (f_V(\alpha_2),f_V(\beta_2))\cdots (f_V(\alpha_k),f_V(\beta_k)).$$ Then, applying $m^{-1}$, notice that $$\{f_V(\alpha_1),f_V(\beta_1)\}\{f_V(\alpha_2),f_V(\beta_2)\}\cdots \{f_V(\alpha_k),f_V(\beta_k)\}$$ forms a cycle of edges in $\HH(K)$. Since $f_V$ is bijective, it follows that \\ $\{\alpha_1, \beta_1\}\{\alpha_2,\beta_2\} \cdots \{\alpha_k, \beta_k\}$  also forms a cycle of edges in $\HH(K)$, implying that $(\alpha_1,\beta_1)(\alpha_1,\beta_1)\cdots (\alpha_k,\beta_k)$ also contains a nontrivial closed $V$-path, yielding a contradiction. We conclude that $g(\sigma) \in \M(K)$, so $g$ is a simplicial map. \\

Since $g_V$ is a bijection and $g$ is simplicial, it follows that $g$ is a simplicial automorphism on $\M(K)$. Define $\phi(f) := g$. We first show that $\phi$ is a homomorphism. Suppose we have $a, b \in \Aut(\Hasse(K))$. We seek to show that $\phi(a\circ b) = \phi(a)\circ \phi(b)$. By the definition of a simplicial map, it suffices to show that $\phi(a\circ b)_V = (\phi(a)\circ \phi(b))_V$. We have
\begin{eqnarray*}
\phi(a\circ b)_V&=& m\circ a\circ b \circ m^{-1}\\
&=& (m \circ a \circ m^{-1}) \circ (m \circ b \circ m^{-1})\\
&=& \phi(a)_V \circ \phi(b)_V.
\end{eqnarray*}
By Lemma \ref{lem: comp}, we know that $\phi(a\circ b)_V = \phi(a)_V \circ \phi(b)_V = (\phi(a)\circ \phi(b))_V$, as desired.\\

To see that $\phi$ is injective, suppose  that $\phi(a) = \phi(b)$. Then we have $m \circ a \circ m^{-1} =  m\circ b \circ m^{-1}$ which implies that $a=b$. \\

Finally, we show that $\phi$ is surjective. For any $g \in \Aut(\M(K))$, $g$ must be induced by some $g_V: V(\M(K)) \to V(\M(K))$. Then construct an $f \colon \Hasse(K)\to \Hasse(K)$ defined by $f = m^{-1} \circ g \circ m$. We see that $\phi(f)  = m\circ (m^{-1} \circ g \circ m) \circ m^{-1} = g$, as desired. Therefore, we conclude that $\Aut(\Hasse(K)) \cong \Aut(\M(K))$.

\end{proof}

We now are able to easily compute the automorphism group of $C_n$.

\begin{prop}\label{prop: aut C_n} If $K = C_n$, then then $\Aut(\M(C_n)) \cong \Aut(C_{2n}) $.
\end{prop}

\begin{proof}
It suffices to show $\Aut(\Hasse(C_n)) \cong \Aut(C_{2n})$.
We construct the Hasse diagram of $C_n$:

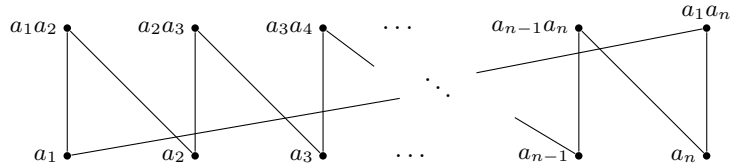
$$
\begin{tikzpicture}[scale=1.7]

\node[inner sep=1pt, circle, fill=black](a1) at (0,0) {};
\node[inner sep=1pt, circle, fill=black](a2) at (1,0) {};
\node[inner sep=1pt, circle, fill=black](a3) at (2,0) {};
\node[inner sep=1pt, circle, fill=black](a4) at (4,0) {};
\node[inner sep=1pt, circle, fill=black](ak) at (5,0) {};
\node at (2.7,0) {\ldots};

\node[inner sep=1pt, circle, fill=black](a1a2) at (0,1) {};
\node[inner sep=1pt, circle, fill=black](a2a3) at (1,1) {};
\node[inner sep=1pt, circle, fill=black](a3a4) at (2,1) {};
\node[inner sep=1pt, circle, fill=black](a4ak) at (4,1) {};
\node[inner sep=1pt, circle, fill=black](a1ak) at (5,1) {};
\node at (2.6,1) {\ldots};
\node at (2.9,0.6) {$\ddots$};

\draw[-]  (a1)--(a1a2); \draw[-] (a1)--(2.6, 0.45); \draw[-] (3.2, 0.65)--(a1ak);
\draw[-] (a2)--(a1a2); \draw[-] (a2)--(a2a3);
\draw[-] (a3)--(a2a3); \draw[-] (a3)--(a3a4);
\draw[-] (ak)--(a1ak);
\draw[-] (a4)--(a4ak);
\draw[-] (ak)--(a4ak);
\draw[-] (a3a4)--(2.4,0.7);
\draw[-] (a4)--(3.5,0.3);

\node[anchor = east]  at (a1) {\small{$a_1$}};
\node[anchor =  east]  at (a2) {\small{$a_2$}};
\node[anchor = east]  at (a3) {\small{$a_3$}};
\node[anchor = east]  at (a4) {\small{$a_{n-1}$}};
\node[anchor = east]  at (ak) {\small{$a_n$}};

\node[anchor = east]  at (a1a2) {\small{$a_1a_2$}};
\node[anchor = south]  at (a1ak) {\small{$a_1a_n$}};
\node[anchor =  east]  at (a2a3) {\small{$a_2a_3$}};
\node[anchor = east]  at (a3a4) {\small{$a_3a_4$}};
\node[anchor = east]  at (a4ak) {\small{$a_{n-1}a_n$}};

\end{tikzpicture}
$$

It is clear that this Hasse Diagram can be redrawn as:
$$
\begin{tikzpicture}

\node[regular polygon,regular polygon sides=13,minimum size=5cm,draw] (a){};
\foreach \x in {9,...,13}{\node[circle,fill,inner sep=1pt] at (a.corner \x) {};}
\foreach \x in {1,...,5}{\node[circle,fill,inner sep=1pt] at (a.corner \x) {};}

\draw [line width=1.1cm, white] (1.1,-2) -- (-2.3,-2);
\node at (-1.71,-1.55) {$\ddots$};
\node at (0.85,-2.17) {$\dots$};
\node[anchor = south]  at (a.corner 1) {\small{$a_1$}};
\node[anchor = east]  at (a.corner 2) {\small{$a_1a_n$}};
\node[anchor = east]  at (a.corner 3) {\small{$a_n$}};
\node[anchor = east]  at (a.corner 4) {\small{$a_{n-1}a_n$}};
\node[anchor = east]  at (a.corner 5) {\small{$a_{n-1}$}};

\node[anchor =  west]  at (a.corner 12) {\small{$a_2$}};
\node[anchor = west]  at (a.corner 10) {\small{$a_3$}};

\node[anchor = west]  at (a.corner 13) {\small{$a_1a_2$}};

\node[anchor =  west]  at (a.corner 11) {\small{$a_2a_3$}};
\node[anchor = west]  at (a.corner 9) {\small{$a_3a_4$}};

\end{tikzpicture}
$$
which is $C_{2n}$. Therefore, $\Aut(\HH(C_n)) \cong \Aut(C_{2n})$.

\end{proof}

\begin{rem} Let $D_{n}$ be the dihedral group of order $n$. It is well known that $D_{2n}\cong D_n \times \mathbb{Z}_2$ for $n$ odd.  Since $\Aut(C_n) \cong D_{2n}$, we have $\Aut(\M(C_n)) \cong \Aut(C_{2n}) \cong \Aut(C_n) \times \mathbb{Z}_2$ whenever $n$ is odd, yielding the same formula as the automorphism group of $\pd^n$ that we show in Section \ref{sec: Morse complex of pd}.
\end{rem}

\subsection{The Morse complex of $\partial\Delta^n$}\label{sec: Morse complex of pd}

We now investigate the case where $K=\pd^n$.  In this case, as in the $K=C_n$ case, there are automorphisms of the Morse complex which are not induced by an automorphism of the original complex.  While these automorphisms of the Morse complex are not induced by simplicial maps, we will show below that they are induced by what we are calling the reflection map. This is not a simplicial map, but rather a ``cosimplicial map," a term we define in Definition \ref{defn: cosimplicial}. The automorphisms induced by the simplicial maps and those induced by this cosimplicial map will then be shown to generate all possible automorphisms of the Morse complex, allowing us to compute and $\Aut(\M(\pd^n)$ in Theorem \ref{prop: aut(M(K)) pt2}.  To illustrate, we first look at an example.

\begin{ex}\label{ex: delta1}  Let $K=\pd^1=C_3$, and recall that we computed the Morse complex of $K$ in Example \ref{ex: s1}. For reference, we give the Morse complex here, noting again the convention that $ab$ is shorthand for the primitive vector $(a,ab)$.

$$
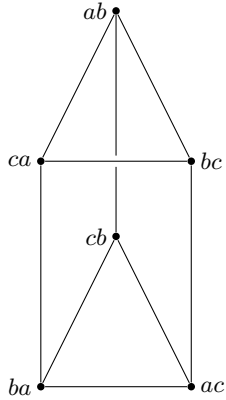
\begin{tikzpicture}[scale=1]

\node[inner sep=1pt, circle, fill=black](a) at (0,2) {};
\node[inner sep=1pt, circle, fill=black](b) at (1,0) {};
\node[inner sep=1pt, circle, fill=black](c) at (-1,0) {};
\node[inner sep=1pt, circle, fill=black](d) at (0,-1) {};
\node[inner sep=1pt, circle, fill=black](e) at (1,-3) {};
\node[inner sep=1pt, circle, fill=black](f) at (-1,-3) {};

\draw[-]  (a)--(b);
\draw[-]  (a)--(c);
\draw[-] (a)--(d);
\draw [white, line width=1.5mm] (b) -- (c) node[midway, left] {};
\draw[-]  (b)--(c);

\draw[-]  (d)--(e);
\draw[-]  (d)--(f);
\draw[-]  (e)--(f);

\draw[-] (b)--(e);
\draw[-] (c)--(f);

\node[anchor = east]  at (a) {\small{$ab$}};
\node[anchor =  west]  at (b) {\small{$bc$}};
\node[anchor = east]  at (c) {\small{$ca$}};
\node[anchor = east]  at (d) {\small{$cb$}};
\node[anchor =  west]  at (e) {\small{$ac$}};
\node[anchor = east]  at (f) {\small{$ba$}};

\end{tikzpicture}
$$
We see that $\Aut(\pd^1)$ is the symmetries of a triangle, and has six automorphisms. Meanwhile, $\Aut(\M(\pd^1))$ has twelve automorphisms. Six of the automorphisms of $\M(\pd^1)$ arise from automorphisms of $\pd^n$, but the other six do not. For example, define $F\colon \M(\pd^1)\to \M(\pd^1)$ by
\begin{eqnarray*}
F(ab)&=&cb\\
F(ba)&=&ca\\
F(ca)&=&ba\\
F(cb)&=&ab\\
F(ac)&=&bc\\
F(bc)&=&ac.
\end{eqnarray*}
Then it is easy to see that $F$ is a simplicial automorphism.  However, it is not induced by any automorphism of $\pd^1$.  Furthermore, composition of $F$ with any automorphism induced by an automorphism of $\pd^1$ yields a new automorphism of $\M(\pd^1)$ that is not induced by an automorphism of $\pd^1$.  For example, if $f\colon \pd^1\to \pd^1$ is given by $f(a)=a, f(b)=c$, and $f(c)=b$, then $F\circ f_*$ is an automorphism of $\M(\pd^1)$ which is not equal to any automorphism induced from $\pd^1$.
\end{ex}

The map $F$ in Example \ref{ex: delta1} is what in general we call $\pi_*$, the induced map of the reflection map $\pi_n \colon \pd^n\to \pd^n$ given below in Definition \ref{defn: reflection map}. We will show in Lemma \ref{lem: ghost2} that $\pi_*$ is a simplicial automorphism of $\M(\pd^n)$ and that it generates all the ``missing" automorphisms of $\M(\pd^n)$ in Theorem \ref{prop: aut(M(K)) pt2}.

\begin{defn}\label{defn: reflection map} Let $\pd^{n}$ be the boundary of the $n$-simplex on the vertices $\{v_0, v_1, \ldots, v_n\}$ and write $\delta:=v_0v_1\cdots v_{n}$.  Define the \textbf{reflection map} $\pi_n=\pi\colon \pd^n \to \pd^n$ by $\pi(\sigma) := \delta - \sigma $.
\end{defn}

The reflection map is a cosimplicial map in the following sense:

\begin{defn}\label{defn: cosimplicial}
Let $K$ be a simplicial complex, $f\colon K \to K$ a function such that if $\sigma$ is a simplex in $K$, $f(\sigma)$ is a simplex of $K$.  Then $f$ is called a \textbf{cosimplicial map} if whenever $\tau\subseteq \sigma$, then $f(\tau)\supseteq f(\sigma).$ If in addition $f$ is a bijection on the simplices of $K$, we say that $f$ is a \textbf{cosimplicial automorphism}.
\end{defn}

\begin{ex}  We now illustrate how the reflection map $K=\pd^2$ gives rise to an automorphism on the Hasse diagram of $\partial\Delta^2$.
\begin{comment}
$$
\begin{tikzpicture}[scale=1.7]

\node[inner sep=1pt, circle, fill=black](a) at (0,0) {};
\node[inner sep=1pt, circle, fill=black](b) at (1,0) {};
\node[inner sep=1pt, circle, fill=black](c) at (2,0) {};
\node[inner sep=1pt, circle, fill=black](d) at (3,0) {};

\node[inner sep=1pt, circle, fill=black](ab) at (0,1) {};
\node[inner sep=1pt, circle, fill=black](bc) at (1,1) {};
\node[inner sep=1pt, circle, fill=black](cd) at (2,1) {};
\node[inner sep=1pt, circle, fill=black](ad) at (3,1) {};

\draw[-] (b)--(ab); \draw[-] (b)--(bc);
\draw[-] (c)--(cd); \draw[-] (c)--(bc);
\draw[-] (d)--(cd); \draw[-] (d)--(ad);
\draw [white, line width=1.5mm] (a) -- (ad) node[midway, left] {};
\draw[-]  (a)--(ab); \draw[-] (a)--(ad);

\node[anchor = east]  at (a) {\small{$a$}};
\node[anchor =  east]  at (b) {\small{$b$}};
\node[anchor = east]  at (c) {\small{$c$}};
\node[anchor = east]  at (d) {\small{$d$}};

\node[anchor = east]  at (ab) {\small{$ab$}};
\node[anchor = south]  at (ad) {\small{$ad$}};
\node[anchor =  east]  at (bc) {\small{$bc$}};
\node[anchor = east]  at (cd) {\small{$cd$}};

\end{tikzpicture}
$$
\end{comment}

$$
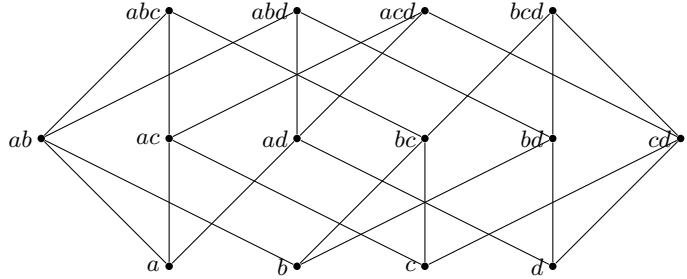
\begin{tikzpicture}[scale=1.7]

\node[inner sep=1pt, circle, fill=black](a) at (0,0) {};
\node[inner sep=1pt, circle, fill=black](b) at (1,0) {};
\node[inner sep=1pt, circle, fill=black](c) at (2,0) {};
\node[inner sep=1pt, circle, fill=black](d) at (3,0) {};

\node[inner sep=1pt, circle, fill=black](ab) at (-1,1) {};
\node[inner sep=1pt, circle, fill=black](ac) at (0,1) {};
\node[inner sep=1pt, circle, fill=black](ad) at (1,1) {};
\node[inner sep=1pt, circle, fill=black](bc) at (2,1) {};
\node[inner sep=1pt, circle, fill=black](bd) at (3,1) {};
\node[inner sep=1pt, circle, fill=black](cd) at (4,1) {};
\node[inner sep=1pt, circle, fill=black](bcd) at (3,2) {};
\node[inner sep=1pt, circle, fill=black](acd) at (2,2) {};
\node[inner sep=1pt, circle, fill=black](abd) at (1,2) {};
\node[inner sep=1pt, circle, fill=black](abc) at (0,2) {};

\draw[-]  (a)--(ab); \draw[-]  (a)--(ac); \draw[-]  (a)--(ad);
\draw[-]  (b)--(ab); \draw[-]  (b)--(bc); \draw[-]  (b)--(bd);
\draw[-]  (c)--(ac); \draw[-]  (c)--(bc); \draw[-]  (c)--(cd);
\draw[-]  (d)--(ad); \draw[-]  (d)--(bd); \draw[-]  (d)--(cd);

\draw[-]  (bcd)--(bc); \draw[-]  (bcd)--(bd); \draw[-]  (bcd)--(cd);
\draw[-]  (acd)--(ac); \draw[-]  (acd)--(ad); \draw[-]  (acd)--(cd);
\draw[-]  (abd)--(ab); \draw[-]  (abd)--(ad); \draw[-]  (abd)--(bd);
\draw[-]  (abc)--(ab); \draw[-]  (abc)--(ac); \draw[-]  (abc)--(bc);

\node[anchor = east]  at (a) {\small{$a$}};
\node[anchor =  east]  at (b) {\small{$b$}};
\node[anchor = east]  at (c) {\small{$c$}};
\node[anchor = east]  at (d) {\small{$d$}};

\node[anchor = east]  at (ab) {\small{$ab$}};
\node[anchor = east]  at (ad) {\small{$ad$}};
\node[anchor = east]  at (ac) {\small{$ac$}};
\node[anchor =  east]  at (bc) {\small{$bc$}};
\node[anchor =  east]  at (bd) {\small{$bd$}};
\node[anchor = east]  at (cd) {\small{$cd$}};

\node[anchor = east]  at (abc) {\small{$abc$}};
\node[anchor =  east]  at (bcd) {\small{$bcd$}};
\node[anchor = east]  at (acd) {\small{$acd$}};
\node[anchor = east]  at (abd) {\small{$abd$}};

\end{tikzpicture}
$$

Observe that $\partial\Delta^2$ demonstrates rotational symmetry about the center of the diagram. (Notice that $K=C_n$ also demonstrates rotational symmetry, as can be seen in the Hasse diagram in Proposition \ref{prop: aut C_n}.) An automorphism of $\pd^1$ induces an automorphism of the Hasse diagram which permutes the $0-$simplices, but the rotational symmetry of the Hasse diagram arises from the reflection map. For instance, the map induced by the reflection map sends $a\leftrightarrow bcd$, $b\leftrightarrow acd$, and so on, which visually rotates the Hasse Diagram upside down. For simplicial complexes other than $K=\partial\Delta^n$ and $C_n$, the Hasse diagram does not exhibit this rotational symmetry. Hence all automorphisms of the Hasse diagrams correspond to permutations of the $0-$simplices, which in turn correspond to automorphisms on the original $K$. (This follows from Proposition \ref{prop: aut(M(K))=aut(K)}.)

\end{ex}

We now give several basic properties of the reflection map.

\begin{lem}\label{lem: ghost1} The reflection map is a cosimplicial automorphism that commutes with all members of $\Aut(K)$.
\end{lem}

\begin{proof} Clearly if $\sigma$ is a simplex in $\pd^n$, then $\delta-\sigma$ is a simplex in $\pd^n$ since $\pd^n$ by definition is made up of all proper, nonempty subsets of $\delta$. If $\sigma\subseteq \tau$, then $\pi(\sigma)=\delta-\sigma \supseteq \delta-\tau=\pi(\tau)$.  Hence $\pi$ is a cosimplicial map. Next notice that we have $\pi(\pi(\sigma)) = \sigma$, so $\pi$ is its own right and left inverse. It follows that $\pi$ is a cosimplicial automorphism.

To see that $\pi$ commutes with all simplicial automorphisms of $\pd^n$, let $f \in \Aut(\pd^n)$.  Then $f$ is induced by a bijection $f_V\colon V(K)\to V(K)$. Consider any simplex $\sigma \in K$. We seek to show that $\pi(f(\sigma))= f ( \pi (\sigma))$.  It thus suffice to show that $\delta - f(\sigma) = f(\delta - \sigma)$. We proceed by subset inclusion.  Consider any vertex $v \in \delta - f(\sigma)$. Since $f$ is an automorphism, we can express $  v = f(w)$ for some vertex $w$. Then we have $f(w) \not \in f(\sigma)$, so $w \not \in \sigma$. Thus, $w \in \delta - \sigma$, so it follows that $v= f(w) \in f(\delta - \sigma)$. Hence $\delta - f(\sigma) \subseteq f(\delta - \sigma)$. For the other direction, consider any $v \in f(\delta - \sigma)$. Again, we can express $v = f(w)$ for some vertex $w$, hence $w \in \delta - \sigma$. Then $w \not\in \sigma$, so $f(w) \not\in f(\sigma)$. So $v = f(w) \in \delta - f(\sigma)$. Therefore, $f(\delta - \sigma) \subseteq \delta - f(\sigma) $. We conclude that $\delta - f(\sigma) = f(\delta - \sigma)$ so that the reflection map commutes with all of $\Aut(\pd^n).$
\end{proof}

As in Proposition \ref{prop: f_* simplicial}, the reflection map induces a function on the Morse complex $\pi_{\ast_V} \colon V(\M(\pd^n)) \to V(\M(\pd^n))$ defined by $\pi_{\ast_V}((\sigma,\tau)) = (\pi(\tau), \pi(\sigma))$. Even though $\pi$ is not a simplicial map, the induced map is a map on the vertex set of $\M(\pd^n)$ which induces a simplicial map on $\M(\pd^n)$.  The following lemma shows that this induced map on the Morse complex behaves in a similar way to the cosimplicial automorphism on $\pd^n.$

\begin{lem}\label{lem: ghost2} Let $\pi_n=\pi\colon \pd^n\to \pd^n$ be the reflection map, and $\pi_{\ast_V} \colon V(\M(\pd^n)) \to V(\M(\pd^n))$ the induced function on the Morse complex. Then $\pi_{\ast_V}$ is a bijection that commutes with all bijections $g_{\ast_V}\colon V(\M(\pd^n)) \to V(\M(\pd^n))$ that are induced by some $g \in \Aut(\pd^n)$ . Moreover, the induced function $\pi_\ast$ is a simplicial map that commutes with all members of $\Aut(\M(\pd^n))$.
\end{lem}

\begin{proof} That $\pi_{\ast_V}$ is a bijection follows from the fact that $\pi_{\ast_V}$ is its own inverse; that is,
$$\pi_{\ast_V}\circ \pi_{\ast_V}((\sigma,\tau)) = \pi_{\ast_V}((\pi(\tau), \pi(\sigma))) = (\pi(\pi(\sigma)),\pi(\pi(\tau))) = (\sigma,\tau).
$$

To see that $\pi_{\ast_V}$ commutes with any induced bijection $g_{\ast_V}$ observe that

\begin{eqnarray*}
\pi_{\ast_V} \circ g_{\ast_V} ((\sigma,\tau)) &=& \pi_{\ast_V} ((g(\sigma), g(\tau)) \\
&=& (\pi\circ g(\tau), \pi\circ g(\sigma))\\
&=&(g\circ \pi(\tau), g\circ \pi(\sigma)) \\
&=& g_{\ast_V}(\pi(\tau), \pi(\sigma))\\
&=& g_{\ast_V}\circ \pi((\sigma,\tau))
\end{eqnarray*}
where Lemma \ref{lem: ghost1} justifies the fact that $\pi$ and $g$ commute.

Since $\pi_{\ast_V} \circ g_{\ast_V}=g_{\ast_V}\circ \pi_{\ast_V}$, they induce the same function on $\M(K)$. Then $\pi_{\ast_V}\circ g_{\ast_V}$ induces $\pi_{\ast} \circ g_\ast$, and $g_{\ast_V}\circ \pi_{\ast_V} $ induces $g_\ast\circ\pi_\ast$, thus $\pi_{\ast} \circ g_\ast = g_\ast\circ\pi_\ast$.

It remains to verify that this is a simplicial map. Consider any $\sigma  = \prod_{i=0}^k (\alpha_i, \beta_i)\in \M(K)$. We seek to show that $\pi_\ast(\sigma) \in \M(K)$. As we already know that $\pi_{\ast_V}$ is a bijection, it remains to show that if $\pi_\ast(\sigma)$ contains incompatible primitive vectors, so does $\sigma$.\\
% these proofs of things inducing simplicial maps are getting repetitive, is there a way to generalize them? like, if the function (not necessarily simplicial map) inducing the simplicial map on M(K) satisfies X property, then it induces simplicial map? but this is tricky because pi_star switches the order of the simplices

\noindent \textbf{Case 1:} There exists two primitive vectors $\pi_\ast((\alpha_i, \beta_i)) = (\pi(\beta_i), \pi(\alpha_i))$ and  $\pi_\ast((\alpha_j, \beta_j)) = (\pi(\beta_j),\pi(\alpha_j))$ of $\pi_\ast(\sigma)$ that are not compatible. Then, $\{\pi(\beta_i), \pi(\alpha_i)\} \cap \{\pi(\beta_j), \pi(\alpha_j)\} \neq \emptyset$. As $\pi$ is bijective, it follows that $\{\beta_i, \alpha_i\} \cap \{\beta_j,\alpha_j\} \neq \emptyset$, so $(\alpha_i,\beta_i)$ and $(\alpha_j, \beta_j)$ are not compatible.\\

\noindent \textbf{Case 2:} There exists a nontrivial $V$-path $$(\pi(\beta_{i_0}), \pi(\alpha_{i_0})),(\pi(\beta_{i_1}), \pi(\alpha_{i_1})) , (\pi(\beta_{i_2}), \pi(\alpha_{i_2})), \dots , (\pi(\beta_{i_m}), \pi(\alpha_{i_m}))$$
where each $\pi(\beta_{i_j})$ is a codimension $1$ face of $\pi(\alpha_{i_{j-1}})$ for each $1\leq j \leq m$, and $\pi(\beta_{i_0})$ is a codimension $1$ face of $\pi(\alpha_{i_m})$. However, since $\pi$ is a cosimplicial automorphism, we must have that each $\alpha_{i_{j-1}}$ is a codimension $1$ face of $\beta_{i_j}$ for $1\leq j \leq m$, and  $\alpha_{i_m}$. is a codimension $1$ face of  $\beta_{i_0}$. However, this would imply that $(\alpha_{i_m}, \beta_{i_m}), (\alpha_{i_{m-1}}, \beta_{i_{m-1}}), (\alpha_{i_{m-2}}, \beta_{i_{m-2}}), \dots , (\alpha_{i_{0}}, \beta_{i_{0}}) \in \sigma $ is a nontrivial closed V-path, and are therefore not compatible.

Thus $\pi_*$ is a simplicial map, and we therefore conclude that $\pi_\ast \in \Aut(\M(K))$.

\end{proof}

We will refer to the simplicial automorphism $\pi_\ast \in \Aut(\M(K))$ induced by  $\pi_{\ast_V}$ (and in general, any automorphism of $\M(\pd^n)$ that is not induced by a simplicial automorphism of $\pd^n$) as a \textbf{ghost automorphism} on $\M(\pd^n)$.  In the proof of Theorem \ref{prop: aut(M(K)) pt2}, we will see that this ghost automorphism generates all the other ghost automorphisms of $\M(\pd^n)$.

Next we will compute the cardinality of $\Aut(\M(\pd^n))$. We first fix some notation and terminology.  Let $\HH_i$ denote the set of nodes of the Hasse diagram that correspond to simplices in $\pd^n$ of dimension $i$, and let $H = \{\HH_0, \HH_1, \dots , \HH_{n-1}\}$. We also define $\HH_{-1} = \HH_{n} = \emptyset$ for convenience. Call $\HH_i$ the \textbf{$i$th layer} of the Hasse diagram of $\pd^n$. Observe that $|\HH_i| = \dbinom{n+1}{i+1}$ for indices $0\leq i \leq n-1$. Abusing language, we will use simplex to mean both a simplex of $\partial\Delta^n$ and the corresponding vertex of $\HH(\partial\Delta^n)$. We also define the \textbf{degree} of layer $\HH_i$, denoted as $\Deg \HH_i$, as $\Deg \sigma$ where $\sigma \in \HH_i$. Note that this is well-defined since for $\pd^n$, it is clear that if $\sigma, \tau\in \HH_i$, then $\Deg\sigma=\Deg\tau$. We say that two layers $\HH_i, \HH_j$ are \textbf{connected} if there exist $\sigma \in \HH_i, \tau \in \HH_j$ that are connected in $\HH(\partial\Delta^n)$. It is clear by the construction of the Hasse diagram that two layers are connected if and only if $i$ and $j$ are consecutive. It is also clear that connectivity of layers is preserved under automorphisms of $\HH(K)$.\\

\begin{lem}\label{lem: order}  $|\Aut(\M(\partial\Delta^n))| = 2|\Aut(\partial\Delta^n)|$.
\end{lem}
\begin{proof}

By Proposition \ref{prop: aut(M(K))=aut(H(K))}, it suffices to show that $|\Aut(\HH(\pd^n))| = 2|\Aut(\pd^n)|$. Consider an arbitrary automorphism $f \in \Aut(\Hasse(\partial\Delta^n))$.
We claim that the image of any $\HH_i$ under $f$ will either be $\HH_i$ or $\HH_{n-i-1}$. We first establish that $f$ takes $\HH_0$ to $\HH_0$ or $\HH_{n-1}$, and then proceed by induction on $n$.  To that end, observe that if $\sigma\in \HH_0$ or $\HH_{n-1}$, then $\Deg \sigma =n$.  If $\sigma\in \HH_j$ for $1\leq j\leq n-2$, then $\sigma$ has $j+1$ faces of dimension $(j-1)$ and $(n+1)-(j+1)$ cofaces of dimension $j+1$.  Hence $\Deg\sigma=(j+1)+(n+1)-(j+1)=n+1.$ Therefore, the only layer with the same degree as $\HH_0$ is $\HH_{n-1}$, so that any automorphism must send a node of $\HH_0$ into either $\HH_0$ or $\HH_{n-1}$.

Next we establish that $f(\HH_0)=\HH_0$ or $\HH_{n-1}$.  Let $ab$ be a $1$-simplex of $\pd^n$ and suppose for the sake of contradiction that $f$ sends  $a, b \in \HH_0$ to two different layers, say $f(a) \in \HH_0$ and $f(b) \in \HH_{n-1}$. Now $ab$ is connected to both $a$ and $b$, and since $f$ is an automorphism, $f(ab)$ must be connected to both $f(a)$ and $f(b)$.  Since $f(a) \in \HH_0$ and $f(ab)$ is connected to $f(a)$, it follows that $f(ab) \in \HH_1$. Similarly, since $f(b)$ and $f(ab)$ are connected, $f(ab) \in \HH_{n-2}$. Thus, $\HH_{n-2} = \HH_1$, so we must have $n=3$. But it is easily seen by inspection of the $n=3$ case that such an automorphism is impossible. Thus $f(\HH_0) = \HH_0$ or $\HH_{n-1}$.

% Then, we have $f(a) < f(ab) < f(b)$. Let $c = f(b) - f(ab)$. Then, $f(a)c < f(b)$. This means that $f(a)c$ is connected to both $f(b)$ and $f(a)$. We have therefore produced two simplices that are connected to both $f(b)$ and $f(a)$. However, there is only one simplex connected to both $a$ and $b$ (namely $ab$), hence a contradiction.

Having established that $f(\HH_0)=\HH_0$ or $\HH_{n-1}$ for any automorphism $f$, we now show by induction that $f(\HH_i)=\HH_i$ or $\HH_{n-i-1}$ for all $1\leq i \leq n-1$. For the first case, suppose that $f(\HH_0) = \HH_0$, and suppose the inductive hypothesis that for some integer $0< k<n-1$, we have $f(\HH_j) = \HH_j$ for all integers $0\leq j\leq k$. We seek to show that $f(\HH_{k+1}) = \HH_{k+1}$. Notice that $\HH_{k+1}$ is connected to $\HH_k$. Thus, $f(\HH_{k+1})$ is connected to $f(\HH_k)$. Additionally, the only layers connected to $\HH_k$ are $\HH_{k-1}$ and $\HH_{k+1}$. Since $f(\HH_k) = \HH_k$ by the inductive hypothesis, we know $f(\HH_{k+1}) \subseteq \HH_{k-1}\cup \HH_{k+1}$. However, by the inductive hypothesis, we know that $f(\HH_{k-1}) = \HH_{k-1}$. Since $f$ is an isomorphism, this means we cannot send any simplices of $\HH_{k+1}$ to $\HH_{k-1}$ under $f$. Therefore, $f(\HH_{k+1}) \subseteq \HH_{k+1} \implies f(\HH_{k+1}) = \HH_{k+1}$.

The case where $f(\HH_0) = \HH_{n-1}$ is  similar. Suppose that for some integer $0< k<n-1$, we have $f(\HH_j) = \HH_{n-j-1}$ for all integers $0\leq j\leq k$. We seek to show that $f(\HH_{k+1}) = \HH_{n-k-2}$. Since $\HH_{k+1}$ is connected to $\HH_k$, $f(\HH_{k+1})$ is connected to $f(\HH_{k}) = \HH_{n-k-1}$. The only layers connected to $\HH_{n-k-1}$ are $\HH_{n-k-2}, \HH_{n-k}$, so $f(\HH_{k+1}) \subseteq \HH_{n-k-2}\cup \HH_{n-k}$. By the inductive hypothesis, $f(\HH_{k-1}) = \HH_{n-k}$. Since $f$ is injective, we cannot send any simplices of $\HH_{k+1}$ to $\HH_{n-k}$. Thus, $f(\HH_{k+1}) \subseteq \HH_{n-k-2}$. We know that
\begin{eqnarray*}
|f(\HH_{k+1})| &=& |\HH_{k+1}|\\
&=& \binom{n+1}{k+2}\\
&=& \binom{n+1}{n-k-1} \\
&=& |\HH_{n-k-2}|.
\end{eqnarray*}
Hence $f(\HH_{k+1}) = \HH_{n-k-2}$, as desired.

Having established the behaviour of each layer under automorphism, we now establish a group action to count $|\Aut(\HH(\pd^n))|$. Define an action of $\Aut(\Hasse(\pd^n))$ on $H = \{\HH_0, \HH_1, \dots , \HH_{n-1}\}$ by

$$(g, \HH_i) \mapsto g\HH_i := \{g(\sigma)\colon \sigma \in \HH_i\}.$$

We verify this is indeed a group action by noting that, $\id_{\Hasse(\pd^n)} (\HH_i) = \HH_i$ for all $\HH_i$, and that if $g,h \in \Aut(\Hasse(\pd^n))$, we have
\begin{eqnarray*}
(gh)(\HH_i)  &=& \{(g\circ h)(\sigma): \sigma\in \HH_i\}\\
&=& g(\{h(\sigma): \sigma \in \HH_i\} )\\
&=& g(h(\HH_i)).
\end{eqnarray*}

By the Orbit Stabilizer theorem, we have $|\Aut(\Hasse(\pd^n))| = |\mathrm{Orb}(\HH_0)||\mathrm{Stab}(\HH_0)|$. Suppose $f \in |\Aut(\HH(\pd^n))|$ fixes $\HH_0$. Then, $f$ is bijective on the set of vertices of $\pd^n$ so that it corresponds with an automorphism of $\pd^n$. Likewise, any automorphism of $\pd^n$ is induced by a bijective map $V(\pd^n)\to V(\pd^n)$, so must correspond with an automorphism on $\HH(\pd^n)$ that fixes $\HH_0$. Therefore, $|\mathrm{Stab}(\HH_0)| = |\Aut(\pd^n)|$.
We also know that $\mathrm{Orb}(\HH_0) = 2$ since automorphisms of the Hasse Diagram send $\HH_0$ to either $\HH_0$ or $\HH_{n-1}$, as we showed earlier in this proof.
We conclude that $|\Aut(\Hasse(\pd^n))| = 2|\Aut(\pd^n)|$, as desired.
\end{proof}

We are now able to compute the automorphism group of the Morse complex in the case where $K=\pd^n$.

\begin{thm} \label{prop: aut(M(K)) pt2}  If $K =\pd^n$, with $n\geq 2$,  then $\Aut(\M(\pd^n)) \cong \Aut(\pd^n) \times \mathbb{Z}_2$.
\end{thm}

\begin{proof}
For any function $f$, write $f^n := \underbrace{f \circ f \circ \cdots \circ f}_{n\text{ } f\text{'s}}$. We construct an isomorphism $\phi \colon  \Aut(\pd^n)\times \mathbb{Z}_2 \to \Aut(\M(\pd^n))$.
For each $(f, i) \in \Aut(K)\times \mathbb{Z}_2$, $i=0,1$, define $\phi((f,i)) := f_\ast \circ \pi_\ast^i$, where $f_\ast$ is the automorphism of $\M(\pd^n)$ induced by $f$, and $\pi_\ast$ is the ghost automorphism induced by the reflection map.
%Then $\phi(f,0)=f_*$, the induced automorphism of $f$, and $(f,1)=f_*\circ \pi_*$ is a ghost automorphism that is not induced by any automorphism on $\pd^n$.

We first show that $\phi$ is a homomorphism. Suppose we have $(f, i), (g, j) \in \Aut(\pd^n)\times \mathbb{Z}_2$. By Lemma \ref{lem: induced comp} and the fact that $\pi_\ast$ commutes with all $f_*$, we have
\begin{align*}\phi((f,i)(g,j)) &= \phi( (f\circ g, i+j) ) \\ &= (f\circ g)_\ast \circ \pi_\ast^{i+j} \\ &= f_\ast \circ g_\ast \circ \pi_\ast^i\circ \pi_\ast^j  \\&= (f_\ast \circ \pi_\ast^i) \circ (g_\ast \circ  \pi_\ast^j )\\& = \phi((f,i))\phi((g,j)),
\end{align*}
as desired.

Next, we will show that $\phi$ is a bijection. We first show that $\phi$ is injective. Suppose we have $\phi((f,i)) = \phi((g,j))$ for some $(f,i),(g,j) \in \Aut(\pd^n)\times \mathbb{Z}_2$. Then, we have $f_\ast \circ \pi_\ast^i = g_\ast \circ \pi_\ast^j$. If $i=j$, then, we have $f_\ast = g_\ast$, so $f=g$ by Proposition \ref{prop: autK subgroup}. We claim that $i\neq j$ is impossible.  Suppose by contradiction that $i=0, j=1$, so that $f_\ast = g_\ast\circ \pi_\ast$. However, this implies that $g_\ast \circ \pi_\ast$ is induced by some simplicial automorphism on $\pd^n$. Consider any $(\sigma,\tau)\in V(\M(\pd^n))$ with $\dim\sigma=0$. We have $f_{\ast}((\sigma,\tau)) = (f(\sigma), f(\tau))$. We know that $\dim \sigma = \dim f(\sigma) = 0$. However, $g_\ast \circ \pi_\ast((\sigma,\tau)) = g_\ast(\pi(\tau),\pi(\sigma)) = (g(\pi(\tau)), g(\pi(\sigma)))$. We know that $\dim \pi(\tau) = n- \dim\tau = n - 1$. Then, $\dim g(\pi(\tau)) = n-1 > 0$, so $\dim f(\sigma) \neq \dim (g(\pi(\tau)))$, a contradiction. Thus, $i\neq j$ is not possible. Hence $\phi$ is injective.

Finally, by Proposition \ref{prop: aut(M(K))=aut(H(K))} and Lemma \ref{lem: order}, we see that $|\Aut(\M(\pd^n))| = |\Aut(\Hasse(\pd^n))|  = 2|\Aut(\pd^n)| = |\Aut(\pd^n)\times \mathbb{Z}_2|$. Since these groups are finite, $\phi$ is a bijection.  We conclude that $\phi$ is an isomorphism.
\end{proof}

Combining Propositions \ref{prop: aut(M(K))=aut(K)}, \ref{prop: aut C_n}, and Theorem \ref{prop: aut(M(K)) pt2} thus yields Theorem \ref{thm: main} as promised.

\providecommand{\bysame}{\leavevmode\hbox to3em{\hrulefill}\thinspace}
\providecommand{\MR}{\relax\ifhmode\unskip\space\fi MR }
\providecommand{\MRhref}[2]{%
  \href{http://www.ams.org/mathscinet-getitem?mr=#1}{#2}
}
\providecommand{\href}[2]{#2}

\end{document}